\documentclass[12pt]{amsart}

\usepackage{amsfonts}
\usepackage{amssymb}
\usepackage{amsthm}
\usepackage{amsmath}
\usepackage{enumerate}
\usepackage{mathtools}
\usepackage{graphicx}
\usepackage{color}

\newtheorem{theorem}{Theorem}[section]
\newtheorem{lemma}[theorem]{Lemma}
\newtheorem{remark}[theorem]{Remark}
\newtheorem{claim}[theorem]{Claim}
\newtheorem{proposition}[theorem]{Proposition}
\newtheorem{conj}[theorem]{Conjecture}

\newtheorem{corollary}[theorem]{Corollary}
\newtheorem{definition}[theorem]{Definition}
\numberwithin{equation}{section}

\newcommand{\A}{{\mathcal{A}}}

\newcommand{\eps}{\varepsilon}

\begin{document}
\title[Free boundary minimal disks in convex balls]{Free boundary minimal disks in convex balls}
\author{Robert Haslhofer}\address{Department of Mathematics\\ University of Toronto\\Toronto, ON M5S 2E4\\ Canada}
\email{roberth@math.toronto.edu}

\author{Daniel Ketover}\address{Department of Mathematics\\ Rutgers University\\Piscataway, NJ 08854\\ USA}
 \email{dk927@math.rutgers.edu}

\begin{abstract} In this paper, we prove that every strictly convex 3-ball with nonnegative Ricci-curvature contains at least 3 embedded free-boundary minimal 2-disks for any generic metric, and at least 2 solutions even without genericity assumption. Our approach combines ideas from mean curvature flow, min-max theory and degree theory. We also establish the existence of smooth free-boundary mean-convex foliations. In stark contrast to our prior work in the closed setting, the present result is sharp for generic metrics.
\end{abstract}
\date{\today}
\maketitle

\tableofcontents

\section{Introduction}

Let us open by recalling the following classical theorem, which confirms a conjecture from 1905 by Poincare \cite{Poincare_geod_conj}:

\begin{theorem}[{geodesics \cite{Birkhoff,LS,Grayson}}]
$S^2$ equipped with an arbitrary Riemannian metric contains at least 3 embedded closed geodesics.
\end{theorem}

The first such geodesic was found by Birkhoff \cite{Birkhoff}, who invented the min-max method. The existence of at least 3 geodesics, leaving open the question of embeddedness, was proved by Lusternik-Schnirelmann  \cite{LS}. Finally, embeddedness was established by Grayson \cite{Grayson}.\\

Moving up one dimension, the two corresponding fundamental conjectures for minimal surfaces, see e.g. \cite{Yau_problems,J,Zhou_ICM}, are:

\begin{conj}[{minimal spheres}]\label{conj1}
$S^3$ equipped with an arbitrary Riemannian metric contains at least 4 embedded minimal 2-spheres.
\end{conj}

\begin{conj}[{minimal disks}]\label{conj2}
Every strictly convex ball in $\mathbb{R}^3$ (or more generally in any 3-manifold with nonnegative Ricci-curvature) contains at least 3 embedded free-boundary minimal 2-disks. 
\end{conj}

Regarding Conjecture \ref{conj1} (minimal spheres), the first minimal 2-sphere was found by Simon-Smith \cite{SS}, who adapted the more general min-max theory of Almgren and Pitts to the case of surfaces with fixed topology. Later, White \cite{White_deg2} used arguments from degree theory to prove the existence of at least $2$ solutions assuming the metric has positive Ricci curvature. More recently, in our paper \cite{HaslhoferKetover} we combined techniques from mean curvature flow and min-max theory to establish the existence of at least $2$ solutions for generic metrics without any assumptions on the sign of the curvature. Finally, in a recent breakthrough Wang-Zhou \cite{WangZhou} proved the multiplicity-one conjecture for unstable Simon-Smith min-max surfaces, which, incorporating also the techniques from our prior paper \cite{HaslhoferKetover} as well as techniques from Song \cite{Song_inf}, enabled them to establish the existence of at least 4 solutions for any generic metric.\\

Regarding Conjecture \ref{conj2} (minimal disks), a first free-boundary disc, leaving open the question of embeddedness, was found by Struwe \cite{Struwe_fb}, who used a parametric approach. A few years later, Gr\"uter-Jost \cite{GruterJost} (see also Jost \cite{Jost2}) found an embedded solution, by adapting the Simon-Smith approach to the free-boundary setting. This min-max approach to constructing free-boundary minimal disks has been further developed over the years, in particular in important papers by Fraser \cite{Fraser_fb}, DeLellis-Ramic \cite{DLR}, Li \cite{Li}, Li-Zhou \cite{LiZhou}, Lin-Sun-Zhou \cite{LSZ} and Laurain-Petrides \cite{LaurainPetrides}, which however are all concerned with getting 1 solution in greater generality rather than  getting more than 1 solution. The only exception to this we are aware of is an interesting paper by Jost \cite{Jost2}, which obtains 3 solutions for domains close to the round ball.\\

Let us recall the heuristics behind these conjectures. Denoting by $\mathcal{G}$ the space of geodesic spheres in the round $3$-sphere, and identifying the degenerate (i.e. point) spheres $\partial \mathcal{G}$, one has $\mathcal{G}/\partial {\mathcal{G}}\cong \mathbb{RP}^4$. Likewise, for the space of flat disks $\mathcal{F}$ in the standard $3$-ball one has $\mathcal{F}/\partial {\mathcal{F}}\cong \mathbb{RP}^3$. Hence, in light of Hatcher's theorem \cite{Hatcher}, one expects that the area functional has at least $5$ or $4$ critical points, respectively. Since the absolute minimum of the area is zero achieved by a point, one thus expects at least $4$ or $3$ nontrivial solutions, respectively. The major difficulty in turning these heuristics into proofs of the conjectures is that solutions with differing integer multiplicities count as distinct critical points of the area functional and thus may not be geometrically distinct.\\

We also mention that initiated with the proof of the Willmore conjecture by Marques-Neves \cite{MN_Willmore}, there have been spectacular developments in Almgren-Pitts and Allen-Cahn min-max theory, including in particular the existence of infinitely many closed minimal surfaces by Marques-Neves \cite{MN_inf} and Song \cite{Song_inf} and infinitely many free-boundary minimal surfaces by Wang \cite{Wang_inf}, the proof of the Weyl law by Liokumovich-Marques-Neves \cite{LMN}, which led to the equidistribution results by Irie-Marques-Neves \cite{IMN} and Marques-Neves-Song \cite{MNS_equi}, the proof of the multiplicity-one conjecture by Chodosh-Mantoulidis \cite{ChodoshMantoulidis} and Zhou \cite{Zhou_mult1} (see also Sun-Wang-Zhou \cite{SWZ_mult1} and Guaraco-Marques-Neves \cite{GMN}), and the Morse index estimates from Marques-Neves \cite{MN_Morse1,MN_Morse2}. However, unfortunately these theories do not provide any information regarding the question whether or not any of the produced solutions are of disk or sphere type.\\

In this paper, we prove Conjecture \ref{conj2} (minimal disks) for generic metrics, and make progress towards the general case as well:

\begin{theorem}[minimal disks] \label{thm_main}
Every strictly convex 3-ball with nonnegative Ricci-curvature contains at least 3 embedded free-boundary minimal disks for any generic metric, and at least 2 solutions even without genericity assumption. Moreover, the area of our 2nd  solution is always strictly less than twice the area of the Gr\"uter-Jost solution.
\end{theorem}

In contrast to our prior paper in the closed case \cite{HaslhoferKetover}, where we were only able to find 2 out of the conjectured 4 minimal 2-spheres on 3-spheres with generic metrics, Theorem \ref{thm_main} (minimal disks) is sharp for generic metrics. Namely, it produces all the 3 conjectured solutions.
 

A natural family of examples of 3-balls to illustrate Theorem \ref{thm_main} (minimal disks) are the ellipsoids
\begin{equation}
E(a,b,c):= \left\{\frac{x^2}{a^2}+ \frac{y^2}{b^2}+\frac{z^2}{c^2}\leq 1\right\}\subset\mathbb{R}^3.
\end{equation}
They contain at least 3 obvious `planar' solutions, which are obtained by intersecting $E(a,b,c)$ with the coordinates planes. On the other hand, for sufficiently elongated ellipsoids we obtain: 

\begin{corollary}[ellipsoids]\label{cor_ell}
For $a\geq 2\max (b,c)$ the ellipsoid $E(a,b,c)$ contains a nonplanar embedded free-boundary minimal disk $\Sigma(a)$. Moreover, for $a\to \infty$ our surfaces $\Sigma(a)$ converge in the sense of varifolds to the planar disk $\{x=0\}\times E(b,c)\subset \mathbb{R}\times E(b,c)$ with multiplicity-two.
\end{corollary}

Indeed, the first part follows from the theorem by observing that for $a\geq 2\max (b,c)$ the areas of $xy$-planar disk and $xz$-planar disk are at least twice as large as the one of $yz$-planar disk. This answers a question by Dierkes-Hildebrandt-K\"uster-Wohlrab \cite{DHKW}. In the circular symmetric case, i.e. for $b=c$, the corollary has been obtained recently by Petrides \cite{Petrides_disks}, using entirely different techniques.\\

Similarly as in our previous paper in the closed case \cite{HaslhoferKetover}, our approach combines ideas from mean curvature flow and min-max theory. Specifically, using the free-boundary flow with surgery from recent work of the first author \cite{Haslhofer_fb_surgery}, we produce smooth free-boundary mean-convex foliations. More precisely, we prove:

\begin{theorem}[free-boundary foliation dichotomy]\label{thm_fol}
Let $M$ be a 3-ball equipped with an arbitrary Riemannian metric with strictly convex boundary, and suppose $\partial K\subset M$ is a smooth free-boundary strictly mean-convex disk. Then, one of the following alternatives holds true:
\begin{enumerate}
\item\label{fol_case1} There exists a stable embedded free-boundary minimal disk $\Sigma\subset\mathrm{Int}_M(K)$ or a stable embedded minimal two-sphere $\Gamma\subset\mathrm{Int}_M(K)\setminus\partial M$.
\item\label{fol_case2} There exists a smooth foliation $\{ \Sigma_t \}_{t\in [0,1]}$ of $K$ by free-boundary strictly mean-convex disks. More precisely:
\begin{itemize}
\item $\{\Sigma_t\}_{t\in [0,1)}$ is a smooth family of properly embedded free-boundary strictly mean-convex disks starting at $\Sigma_0=\partial K$.
\item $\cup_{t\in [0,1]}\Sigma_t = K$, and $\Sigma_{t_1}\cap \Sigma_{t_2}=\emptyset$ whenever $t_1\neq t_2$.
\item $\Sigma_t$ for $t\to 1$ Hausdorff converges to a point on $\partial M$.\footnote{One can arrange that the foliation terminates in a round half-point.}
\end{itemize}
\end{enumerate}
\end{theorem}

We note that in our application we are always in case (\ref{fol_case2}). We have stated Theorem \ref{thm_fol} (free-boundary foliation dichotomy) in its more general form since it is clearly of independent interest, and will likely have several other applications in the geometry of three-manifolds, possibly even in the study of inverse problems \cite{ABN}. We also note that the flow with surgery at first only produces a ``foliation with gaps", caused by surgeries and discarding, so we have to fill in these gaps using a strategy as in the work of Buzano, Hershkovits and the first author \cite{BHH}, which in turn was inspired by important work by Marques \cite{Marques_Ricci_moduli}.\\

To outline our proof of Theorem \ref{thm_main} (minimal disks), let us fix any compact connected 3-manifold $M$ with nonnegative Ricci curvature and nonempty strictly convex boundary (it follows from the work of Meeks-Simon-Yau \cite{MSY} that any such manifold is indeed diffeomorphic to the 3-ball). Recall that Gr\"uter-Jost \cite{GruterJost} already proved the existence of at least 1 solution. Moreover, by a beautiful degree theory argument of Maximo-Nunes-Smith \cite{MNS,MNS2} for generic metrics the number of solutions is always odd. Hence, our task is to produce a 2nd solution.\\

To get started, sliding the Gr\"uter-Jost disk a bit to both sides we can decompose $M=K^-\cup Z \cup K^+$, where $Z$ is a short cylindrical region and $\partial K^\pm$ are smooth strictly mean-convex disks with free-boundary.
Applying Theorem \ref{thm_fol} (free-boundary foliation dichotomy) we can build an optimal free-boundary foliation of $M$, namely a foliation $\{\Sigma_t\}_{t\in [-1,1]}$ of $M$ by free-boundary disks, such that the Gr\"uter-Jost disk sits in the middle of the foliation as $\Sigma_{0}$ and all other slices have strictly less area.\\

Using our optimal foliation we can then form a certain two parameter family $\{\Sigma_{s,t}\}$. Loosely speaking, this family is constructed by joining the surfaces $\Sigma_s$ and $\Sigma_t$ by a thin half neck.  
Establishing a half version of the catenoid estimate from Marques, Neves and the second author \cite{KMN}, we can suitably open up the half neck to arrange that
\begin{equation}\label{upperbound}
\sup_{s,t} |\Sigma_{s,t}| < 2|\Sigma_{0}|.
\end{equation}
This guarantees that min-max for our two-parameter family does not simply produce the Gr\"uter-Jost disk with multiplicity-two, and together with a standard Lusternik-Schnirelmann argument to handle the case when the min-max value for our two-parameter family equals the area of the Gr\"uter-Jost disk, allows us to conclude.\\

Finally, let us comment on potential other approaches towards Conjecture \ref{conj2} (minimal disks). Given the recent breakthrough in the closed case by Wang-Zhou \cite{WangZhou}, it seems likely that their method will generalize to the free-boundary setting as well. Their approach has the advantage that it potentially also applies to other related problems, while our approach has the advantage that it in addition produces smooth free-boundary foliations. In a different direction, there is a correspondence between free-boundary minimal disks and half-harmonic geodesics as first pointed out by DaLio-Martinazzi-Riviere \cite{DMR} and Millot-Sire \cite{MillotSire}, which suggests a more analytic (nonlocal PDE) construction of free-boundary minimal disks. Furthermore, the corresponding half-harmonic gradient flow has been explored in recent interesting work by Wettstein \cite{Wett1,Wett2} and Struwe \cite{Struwe_half}.\\

\bigskip

{\bf Acknowledgements:} R.H. was partially supported by an NSERC discovery grant and a Sloan Research Fellowship. D.K. was partially supported by NSF grant DMS-1906385.\\

\bigskip

\section{The min-max argument}

Throughout this section, $(M,g)$ denotes a 3-ball with nonnegative Ricci curvature and strictly convex boundary. We recall the following basic properties (see e.g. \cite[Section 2]{FraserLi}), which will be used frequently:

\begin{itemize}
\item Every free-boundary minimal surface in $M$ is unstable.
\item There are no closed minimal surfaces in $\mathrm{Int}(M):=M\setminus \partial M$.
\item Any two free-boundary minimal surfaces in $M$ intersect.
\end{itemize}

Indeed, to see this, recall that the second variation of area at a free-boundary minimal surface $\Sigma\subset M$ is given by the quadratic form
\begin{align}\label{2nd_var}
Q_{\Sigma}[\phi]=\int_{\Sigma} |\nabla \phi|^2 -  \left( |A|^2+\mathrm{Ric}(\nu,\nu)\right)\phi^2-\int_{\partial \Sigma}  h(\nu,\nu)\phi^2,
\end{align}
where $h$ denotes the second fundamental form of $\partial M$. Since $h>0$ and $\mathrm{Ric}\geq 0$ by assumption, we see that $Q_{\Sigma}[1]<0$. This shows that there are no stable free-boundary minimal surfaces in $M$. Next, if there was a closed minimal surfaces in $\mathrm{Int}(M)$, then at a point that minimizes the distance to $\partial M$ we would get a contradiction with the second variation of length formula (see e.g. \cite[Appendix A]{Wang_math_ann}). Finally, if the Frankel property failed then thanks to the instability we could push one surface towards the other (see below) and obtain a contradiction with the avoidance principle under free-boundary level set flow \cite{EHIZ}.\\

Also note that any embedded free-boundary minimal surface $\Sigma\subset M$ is properly embedded, i.e. satisfies $\partial\Sigma=\Sigma\cap \partial M$. Indeed, since $\partial M$ has strictly positive mean curvature it cannot touch any interior point of $\Sigma$.\\

\subsection{Sweepouts by disks}
We now discuss min-max theory applied to the free-boundary setting. 
Suppose $\{\Sigma_t\}_{t\in I^n}$, where $I^n=[-1,1]^n$, is a family of closed sets contained in $M$. 

\begin{definition}[sweepout by disks]
We say $\{\Sigma_t\}_{t\in I^n}$ is a \emph{sweepout of $M$ by disks} if there exists a subset $E\subset\partial I^n$, such that:
\begin{enumerate}
\item $\mathcal{H}^2(\Sigma_t)$ is a continuous function of $t\in I^n$
\item  $\Sigma_t$ converges to $\Sigma_{t_0}$ in the Hausdorff topology as $t\rightarrow t_0$.  
\item For $t\in I^n\setminus E$, $\Sigma_t$ is a smooth properly embedded disk in $M$, i.e. $\partial\Sigma_{t}=\Sigma_t\cap \partial M$, that depends smoothly on $t\in I^n\setminus E$.
\item For $t\in E$, $\Sigma_t$ consists of finitely many points and arcs together (possibly) with a smooth properly embedded disk.
\end{enumerate}
\end{definition} 

Given a sweepout of $M$ by disks $\{\Sigma_t\}_{t\in I^n}$ we define its saturation by
\begin{equation}
\Pi:= \big\{ \psi_t(\Sigma_t)\, | \, \{\psi_t\}_{t\in I^n} \, \textrm{is an isotopy with} \, \psi_t =\mathrm{id} \, \mathrm{for} \, t\in \partial I^n \big\}.
\end{equation}
The min-max width of $\Pi$ is then defined as
\begin{equation}
\omega_\Pi=\inf_{\{\Gamma_t\} \in\Pi}\sup_{t\in I^n} |\Sigma_t|,
\end{equation}
where $|\Sigma_t|:= \mathcal{H}^2(\Sigma_t)$. We recall the following foundational theorem:

\begin{theorem}[min-max theorem \cite{GruterJost}, \cite{Jost2}, \cite{Franz}]\label{gj}
Let $\{\Sigma_t\}_{t\in I^n}$ be a sweepout of $M$ by disks and let $\Pi$ be its saturation. Suppose that
\begin{equation}
\omega_\Pi>\sup_{t\in\partial I^n}|\Sigma_t|.
\end{equation}
Then, there exists a suitable minimizing sequence $\{\Sigma_t^j\} \in \Pi$ and a sequence $t_j\in I^n$ with  $|\Sigma^j_{t_j}|\to \omega_\Pi$, such that
$\Sigma^j_{t_j}$ converges in the sense of varifolds to 
 an embedded free-boundary minimal disk $\Sigma\subset M$ with positive integer multiplicity $k$. In particular,
 \begin{equation}
\omega_\Pi = k|\Sigma|.
\end{equation}
Moreover,
\begin{equation}\label{indexbound}
\mbox{index}(\Sigma)\leq n.
\end{equation}
\end{theorem}
Readers who prefer a single reference for the proof of the whole theorem can apply the min-max theorem of Franz \cite[Theorem 1.10]{Franz} in the special case where the equivariance group is trivial.\footnote{At first one only obtains a potentially disconnected union of minimal disks, but by the Frankel property the limiting stationary varifold must be connected.}
 Historically, except for the index bounds which are due to Franz, the theorem was first obtained by Gr\"uter-Jost \cite{GruterJost} and Jost \cite{Jost2}, who generalized the theory of Simon-Smith \cite{SS} to the setting with boundary.\\
 
In particular, for any diffeomorphism $\psi$ from $M$ to the closed unit ball $B\subset \mathbb{R}^3$, one can consider the fundamental 1-parameter sweepout
\begin{equation}
\Sigma_t = \psi^{-1}(B\cap \{x_3 = t\}),
\end{equation}
where $t\in [-1,1]$. Denoting by $\Pi_1$ the saturation of this sweepout, thanks to the isoperimetric inequality its min-max width satisfies
\begin{equation}\label{onesweepouts}
\omega_1>0.
\end{equation}
Considering this 1-parameter sweepout Gr\"uter-Jost were able to obtain the existence of at least 1 embedded free-boundary minimal disk.   

\subsection{Optimal foliations} Let us introduce the following notion:

\begin{definition}[optimal foliation] An \emph{optimal foliation} of $M$ is a 1-parameter family of subsets $\{\Sigma_t\}_{t\in I}$, such that
\begin{enumerate}
\item $\{\Sigma_t\}_{t\in \mathrm{Int}(I)}$ is a smooth family of properly embedded free-boundary disks,
\item $\Sigma_{\pm 1}=\{p_{\pm}\}$ for some $p_{\pm}\in \partial M$, and $\Sigma_t$ for $t\to \pm 1$ converges to $\Sigma_{\pm}$ in the Hausdorff sense and to $0$ in the sense of varifolds,
\item $\cup_{t\in I}\Sigma_t=M$, and $\Sigma_{t_1}\cap\Sigma_{t_2}=\emptyset$ whenever $t_1\neq t_2$,
\item $\Sigma_0$ is an embedded free-boundary minimal disk that satisfies $|\Sigma_0|=\omega_1$,
\item $|\Sigma_t| \leq |\Sigma_0| - ct^2$  for all  $t\in I$, where $c=c(M)>0$.\label{a}
\end{enumerate}
\end{definition}

As a consequence of Theorem \ref{thm_fol} (free-boundary foliation dichotomy), which will be proved independently in the next section, we obtain:

\begin{proposition}[optimal foliation]\label{optimalfoliation}
$M$ admits an optimal foliation.  
\end{proposition}

\begin{proof}
As recalled above, there exists some embedded free-boundary minimal disks in $M$ by the result of Gr\"uter-Jost, i.e. by Theorem \ref{gj} (min-max theorem) applied to the fundamental 1-parameter sweepout. Together with the compactness result of Fraser-Li \cite{FraserLi} it follows that there exists some embedded free-boundary minimal disk $\Sigma\subset M$ with least area.\\
 
 To proceed, let us rewrite \eqref{2nd_var} in the form
\begin{align}\label{2nd_var_rew}
Q_{\Sigma}[\phi]=-\int_{\Sigma} \phi L_{\Sigma} \phi  +\int_{\partial \Sigma} \phi(\nabla_n \phi -  h(\nu,\nu)\phi),
\end{align}
where $n$ denotes the outwards unit conormal, and where
\begin{equation}
L_{\Sigma}=\Delta+|A|^2+\mathrm{Ric}(\nu,\nu).
\end{equation}
Consider the lowest eigenvalue with Robin boundary condition, namely
\begin{equation}\label{rayleigh}
\lambda :=\inf \left\{  Q_{\Sigma}[\phi]\, : \, \int_{\Sigma} \phi^2=1   \mbox{ and } \nabla_n \phi =  h(\nu,\nu)\phi  \mbox{ on } \partial \Sigma  \right\}.
\end{equation}
Since $\Sigma$ is unstable, we have $\lambda < 0$. By standard elliptic theory there exists a unique Robin ground state, i.e. a smooth positive function $\phi:\Sigma\to \mathbb{R}_{+}$ normalized such that $\int_{\Sigma}\phi^2=1$, that solves 
\begin{equation}
L_{\Sigma}\phi= \lambda\phi \mbox{ in } \partial\Sigma\quad  \mbox{ and} \quad
\nabla_n\phi = h(\nu, \nu)\phi \mbox{ on } \partial\Sigma. 
\end{equation}
Indeed, considering a minimizing sequence for \eqref{rayleigh} one can first establish the existence of a nonnegative solution, and then the strong maximum principle yields strict positivity and uniqueness.\\

Now, extend $\phi\nu$ to an admissible vector field $X$, namely a vector field on $M$  satisfying $X(p)\in T_p \partial M$ for $p\in \partial M$, and let $\psi_t$ be its flow. Then, setting $\Sigma_t:=\psi_t(\Sigma)$, for $|t|$ small enough we can expand
\begin{equation}\label{taylor1}
H_{\Sigma_t} = t|\lambda|\phi+O(t^2),
\end{equation}
and
\begin{equation}\label{taylor2}
|{\Sigma_t}| =|{\Sigma_0}| -\frac{1}{2}|\lambda|t^2+O(t^3).
\end{equation}
Hence, choosing $\eps>0$ small enough, $\{\Sigma_t\}_{t\in [-\eps,\eps]}$ can be extended to an optimal foliation of $M$ thanks to Theorem \ref{thm_fol} (free-boundary foliation dichotomy). In particular, note that we indeed have $|\Sigma_0|=\omega_1$, since $\Sigma_0$ is realized as maximal area slice of our foliation.\footnote{Moreover, this also implies that $\Sigma_0$ is realized by 1-parameter min-max with multiplicity 1, which sweepingly justifies the name ``Gr\"uter-Jost disk".}
\end{proof}

\subsection{Half-catenoid estimate}

In this subsection, we prove a half-catenoid estimate, by adapting the argument from \cite{KMN} to the setting with boundary.
In contrast to \cite{Wang_math_ann}, where Wang implemented the catenoid estimate for sweepouts by Caccioppoli sets, c.f. \cite{DLR}, in our context we need to be careful to preserve the topological type of the sweepout.  In particular, we have to use half-necks instead of necks and we cannot let them close up too much as otherwise the strict convexity of the boundary would disconnect the surfaces into multiple pieces.\\

Let $\{\Sigma_t\}_{t\in [-1,1]}$ be an optimal foliation of $M$. Following two points $p,q\in \Sigma_0\cap\partial M$ we get two smooth embedded arcs $\alpha = \{\alpha_t\}_{t\in [-1,1]}$ and $\beta = \{\beta_t\}_{t\in [-1,1]}$ in $\partial M$, with $\alpha_0=p$ and $\beta_0=q$ and $\alpha\cap\beta =\{\alpha_{-1},\alpha_{1}\}$, such that $\alpha\cap\Sigma_t = \{\alpha_t\}$ and $\beta\cap\Sigma_t = \{\beta_t\}$ for all $t\in[-1,1]$.\\

Note that we can retract $M\setminus \alpha$ to $\beta$ preserving the leaves of our foliation. More precisely, we can choose a smooth retraction map
\begin{equation}
R:M\setminus\alpha\times [0,1]\to M\setminus\alpha,
\end{equation}
such that setting $\Sigma_t' = \Sigma_t\setminus\alpha_t$, $t\in(-1,1)$, for some constants $t_0>0$ and $\mu_0>0$, and some smooth function $\sigma(t,\mu)\geq 0$, we have
\begin{enumerate}
\item $R(\Sigma_t',\mu_2)\subseteq R(\Sigma_t',\mu_1)\subseteq \Sigma_t'$ for all $t$ and all $\mu_2\geq\mu_1$.\label{ret1}
\item $R(x,0)=x$ for all $x\in M\setminus\alpha$, and $R(\Sigma_t',1)=\beta_t$ for all $t$.\label{ret2}
\item $R(\Sigma_t',\mu) = \{\exp_x(t\phi(x)\nu(x))\in M\;|\; x\in \Sigma\setminus B_{\sigma(t,\mu)}(p)\}$ for $|t|\leq t_0$ and $\mu\leq\mu_0$.\label{ret3}
\item $\tfrac{1}{2}\pi\mu\leq |\partial R(\Sigma_t',\mu)|\leq 2\pi\mu$ for $|t|\leq t_0$ and $\mu\leq\mu_0$. \label{ret4}
\item $\tfrac{1}{4}\pi\mu^2\leq |\Sigma_t\setminus R(\Sigma_t',\mu)|\leq \pi\mu^2$ for $|t|\leq t_0$ and $\mu\leq\mu_0$. \label{ret5}
\end{enumerate}
Here, we arranged that for $|t|$ and $\mu$ small enough our ``half-disks"
\begin{equation}
D(t,\mu) := \Sigma_t\setminus R(\Sigma_t',\mu)
\end{equation}
are close to a Euclidean half-disks. Let us define the ``half-necks"
\begin{equation}
N(s,t,\mu) := \bigcup_{r\in [s,t]}\partial R(\Sigma_r',\mu).
\end{equation}
The following is the main result of this subsection:

\begin{theorem}[half-catenoid estimate]\label{halfcatenoid} 
There exist constants $\eps_1>0$ and $\kappa:=\kappa(\eps_1)>0$ with the following significance. For any $\eps>0$ and $\delta$ satisfying $[\delta, \delta+\eps]\subset [-\eps_1,\eps_1]$ and any $\rho=\rho(\eps,\delta)>0$ small enough there exists a sweepout $\{\Gamma_t\}_{t\in[0,1]}$ of $M$ by disks, such that
\begin{enumerate}
\item $\Gamma_0 =\Sigma_{\delta}\setminus D(\delta,\rho) \cup\Sigma_{\delta+\eps}\setminus D(\delta+\eps,\rho) \cup N(\delta,\delta+\eps,\rho)$,
\item the disk $\Gamma_t$ converges for $t\to 1$ to the point $\beta_\delta$ in the Hausdorff sense and to zero in the varifold sense,
\item $|\Gamma_t|\leq|\Sigma_0|-\kappa(\delta^2+(\delta+\eps)^2)$ for all $t\in[0,1]$.
\end{enumerate}
Moreover, we can arrange that the family varies smoothly if we vary $\delta$.
\end{theorem}

\begin{proof}Recalling that we have fixed $p\in\partial\Sigma_0$, let us abbreviate $B_t:=\Sigma_0\cap B_t(p)$. Setting $r(x) :=\mbox{dist}_{\Sigma_0}(x,p)$, for any $t\in [0,t_0]$ we work with the logarithmic cut-off function 
  \begin{equation}
    \eta_{t}(x) =
    \begin{cases*}
    0& if $r(x)\leq t^2$\\
      (\log t^2 - \log r)/\log(t)     & $t\leq  r(x)\leq t^2$ \\
	      1& if $r(x)\geq t$.
    \end{cases*}
  \end{equation}
Set $\phi_t = \phi\eta_t$, where $\phi$ denotes the Robin ground state of $L_{\Sigma_0}$ as above. For each $t\in [0,t_0]$ and $|h|$ small enough consider the surface $\Lambda_{h,t} = \Lambda^S_{h,t}\cup \Lambda^C_{h,t}$, which we define as union of the catenoidal part
\begin{equation}
\Lambda^C_{h,t} = \{\exp_x(h\phi_t(x)\nu(x))\in M \; |\; x\in B_{t}\setminus B_{t^2} \},
\end{equation}
and the sheet part
\begin{equation}
\Lambda^S_{h,t} = \{\exp_x(h\phi_t(x)\nu(x))\in M \; |\; x\in \Sigma_0\setminus B_t\}.
\end{equation}
By \cite[Proposition 2.5]{KMN} the area of the catenoidal part satisfies
\begin{equation}\label{expansion}
|\Lambda^C_{h,t}| \leq |B_t\setminus B_{t^2}| + h^2\int_{B_t\setminus B_{t^2}} |\nabla\phi_{t}|^2 + h^3\int_{B_t} (1+ |\nabla\phi_{t}|^2).
\end{equation}
Now, thanks to item \eqref{ret4} and our choice of cut-off function, choosing $t_0$ sufficiently small, for all $t\in [0,t_0]$ we have
\begin{equation}\label{log}
\int_{B_t\setminus B_{t^2}} |\nabla\phi_{t}|^2\leq \frac{D}{|\log t|},
\end{equation}
where $D=D(\sup \phi)<\infty$. Hence, for all small enough $|h|$ we get
\begin{equation}\label{catenoidbound}
|\Lambda^C_{h,t}| \leq |B_t\setminus B_{t^2}| + \frac{D}{2|\log t|}h^2.
\end{equation}
For the sheet part, recalling that $\phi$ is an eigenfunction of $L_{\Sigma_0}$ with eigenvalue $\lambda<0$ and satisfies the Robin boundary condition, for all $t$ and $|h|$ small enough we can estimate
\begin{equation}\label{sheetbound}
|\Lambda^S_{h,t}|\leq |\Sigma_0\setminus B_{t}|- \frac{1}{4}|\lambda| h^2.
\end{equation}
Combining the above, possibly after decreasing $t_0$, we thus obtain
\begin{equation}\label{finalcatenoid}
|\Lambda_{h,t}|\leq |\Sigma_0\setminus B_{t^2}| - \frac{1}{8}|\lambda| h^2.
\end{equation}
for all $t\in [0,t_0]$ and $|h|\leq h_0$.\\

We consider the case $\delta<0<\delta+\eps$ (as the case where $\delta$ and $\delta+\eps$ have the same sign is analogous).  Let us set
\begin{equation}
S^{\delta,\delta+\eps}_t:=\Lambda_{\delta,t}\cup\Lambda_{\delta+\eps,t}.
\end{equation}
Then, by \eqref{finalcatenoid},  for all $t\in [0,t_0]$  we have
\begin{equation}\label{areacontrol}
|S^{\delta,\eps}_t| \leq 2|\Sigma_0\setminus B_{t^2}|-\frac{|\lambda|}{8}(\delta^2+(\delta+\eps)^2).
\end{equation}
For small $t$,  however, $S^{\delta,\delta+\eps}_t$ may fail to be a disk.  To remedy this,  we will amend the family so that its level sets $S^{\delta,\delta+\eps}_t\cap\Sigma_s$ never shrink past the half-circles of radius $\rho=\rho(\eps,\delta)>0$.  To this end, assuming that $\eps_1\leq t_0$, for each $t\in [0,t_0]$ we make the following change to $S^{\delta,\delta+\eps}_t$:

For each $\delta<s< \delta+\eps$,  if $S^{\delta,\eps}_t\cap\Sigma_s$ is empty or not contained in $R(\Sigma_s',\rho)$, then we replace $S^{\delta,\eps}_t\cap\Sigma_s$ with the arc $\partial D(s,\rho)$.  For $s=\delta$ and $s=\delta+\eps$,  if $S^{\delta,\eps}_t\cap\Sigma_s$ is empty or not contained in $R(\Sigma_s',\rho)$, then we replace $S^{\delta,\eps}_t\cap\Sigma_s$ with $R(\Sigma_s', \rho)$.   In light of item \eqref{ret3} in the list of properties of the retractions $R$,  the amended $S^{\delta,\eps}_t$ is a disk with boundary contained in $\partial M$.  Moreover,  by the co-area formula and item \eqref{ret4}, the area added by this modification is at most $C\eps_1\rho$.   Choosing $\rho=\rho(\eps,\delta)>0$ sufficiently small, our modified family thus satisfies
\begin{equation}\label{areacontrol}
|S^{\delta,\eps}_t| \leq 2|\Sigma_0\setminus B_{t^2}|-\frac{|\lambda|}{16}(\delta^2+(\delta+\eps)^2).
\end{equation}

Finally, if $\eps_1$ is small enough, it follows from equations 2.58--2.65 in \cite{KMN} that using the retractions $R$ we may extend $\{S^{\delta,\delta+\eps}_t\}_{t\in [0,t_0]}$ to a sweepout $\{S^{\delta,\delta+\eps}_t\}_{t\in [0,1]}$ by disks ending at the desired point.
\end{proof}

\subsection{Two-parameter sweepout} In this subsection, we will construct a suitable 2-parameter sweepout. To begin with, let us introduce the following notion to capture the topological nature of our sweepout:

\begin{definition}[flipping 2-sweepout; c.f. \cite{Ketover_flip}]\label{def_flip}
A \emph{flipping 2-sweepout} associated to an optimal foliation $\{\Sigma_t\}_{t\in [-1,1]}$ of $M$ is a 2-parameter sweepout $\{\Sigma_{s,t}\}_{(s,t)\in [-1,1]^2}$ of $M$ by disks, such that:
\begin{enumerate}
\item For each $s\in(-1,1)$, $\{\Sigma_{s,t}\}_{t\in [-1,1]}$ is a fundamental $1$-sweepout.\label{def_flip1}
\item Up to arcs we have $\Sigma_{-1,t}=\Sigma_t$, $\Sigma_{1,t} =\Sigma_{-t}$, and $\Sigma_{s,\pm 1}=0$.\label{def_flip2}
\end{enumerate}
\end{definition}

\begin{proposition}[flipping 2-sweepout]\label{prop_flip_sweep}
Associated to any optimal foliation $\{\Sigma_t\}_{t\in [-1,1]}$ of $M$ there exists a flipping 2-sweepout $\{\Sigma_{s,t}\}_{(s,t)\in [-1,1]^2}$, such that
\begin{equation}
\sup_{s,t\in [-1,1]^2} |\Sigma_{s,t}| < 2 |\Sigma_0|.
\end{equation}
\end{proposition}

\begin{proof}Consider the triangle $T:=\{-1\leq t\leq s\leq 1\}$. Using the notation from the previous subsection, for $(s,t)\in T$ we set
\begin{equation}
\Gamma'_{s,t}:=R(\Sigma'_s,\phi(s,t))\cup R(\Sigma'_t,\phi(s,t))\cup N(s,t,\phi(s,t)).
\end{equation}
Here, to specify a suitable neck radius function $\phi(s,t)\geq 0$,  note that for all $\gamma>0$ there exists $\beta_1(\gamma),\beta_2(\gamma)>0$, such that
\begin{equation}
\sup_{|s-t|\leq \beta_1(\gamma)}|N(s,t,\cdot)|+  \sup_{\mu \leq \beta_2(\gamma)}|N(\cdot,\cdot,\mu)|< \gamma.
\end{equation}
Given $\mu>0$, we set $\eps:=\mbox{min}(\mu/2, \beta_1(c\mu^2/8))$, and consider the $\eps$-diagonal $D_\eps = \{(s,t)\in T\; |\; s -t\leq\eps\}$.
On $D_\eps$ we define $\phi$ such that $\phi(s,t) = 1$ for $s=t$ and $\phi(s,s-\eps)=\rho(\eps,s)$ for $|s|\leq \mu$, and on $T\setminus D_\eps$ we may take $\phi(s,t)$ to be any nonnegative function less than $\beta_2(c\eps^2/8)$ that vanishes on $\partial T\setminus D_\eps$.

For $(s,t)\in T\setminus D_\eps$,  one of $|t|$ or $|s|$ is at least $\eps/2$, so by the definition of optimal foliation we obtain $|\Sigma_s|+|\Sigma_t|\leq 2|\Sigma_0|-c\eps^2/4$, hence
\begin{equation}
|\Gamma'_{s,t}|\leq 2|\Sigma_0|-c\eps^2/8.
\end{equation}
For $(s,t)\in D_\eps\cap \{ |s|\geq \mu\}$, arguing similarly we infer that
\begin{equation}
|\Gamma'_{s,t}|\leq 2|\Sigma_0|-c\mu^2/8.
\end{equation}

For $(s,t)\in B_\eps:=D_\eps\cap \{ |s|< \mu\}$ we adjust the family $\Gamma'_{s,t}$ using the half-catenoid estimate.  To this end, note that $\{(\delta,\delta-\eps)\}_{|\delta|\leq\mu}$ spans the bottom segment of the quadrilateral $\partial B_\eps$ and that $\Gamma'_{\delta,\delta-\eps} = \Sigma_{\delta}\cup\Sigma_{\delta-\eps}$ up to arcs.  
For each $|\delta|\leq \mu/2$ we replace $\{\Gamma'_{\delta, t}\}_{t\in[\delta-\eps,\delta]}$
with the family provided by Theorem \ref{halfcatenoid} (half-catenoid estimate), reparametrized such that $t$ is in $[\delta-\eps,\delta]$ instead of $[0,1]$. Choosing a suitable interpolation between this replacement and $\Gamma'_{s,t}$ in the region $D_\eps\cap \{\mu/2 < |s|< \mu\}$ we thus obtain a 2-parameter family $\{\Gamma_{s,t}\}_{(s,t)\in T}$, such that
\begin{equation}
\sup_{(s,t)\in T} |\Gamma_{s,t}| < 2 |\Sigma_0|.
\end{equation}
Regarding the values on $\partial T$, observe that up to arcs we have $\Gamma_{s,-1}=\Sigma_s$, $\Gamma_{1,t}=\Sigma_t$ and $\Gamma_{s,s}=0$. Also observe that by construction if we restrict our family to any straight segment joining the corner $(1,-1)$ and any point on the diagonal $\partial T\cap \{s=t\}$, then we get a fundamental 1-sweepout.
  Finally, since $\Gamma_{1,-1}$ is trivial, we can algebraically blowup the corner point $(1,-1)\in T$ to a segment, similarly as in \cite{Ketover_flip}, to obtain a sweepout $\{\Sigma_{\tilde{s},\tilde{t}}\}_{(\tilde{s},\tilde{t})\in [-1,1]^2}$ with the desired properties.
\end{proof}

\subsection{Conclusion of the argument}
We can now prove Theorem \ref{thm_main} (minimal disks), which we restate here in a technical more precise form:

\begin{theorem}[minimal disks]
Any compact connected 3-manifold $(M,g)$ with nonnegative Ricci curvature and strictly convex boundary contains at least $2$ embedded free-boundary minimal disks $\Gamma_1,\Gamma_2\subset M$.

Moreover, we can always choose these solutions such that $\Gamma_1$ has least area among all solutions, $1=\mathrm{Ind}(\Gamma_1)\leq\mathrm{Ind}(\Gamma_2)\leq 2$, and $|\Gamma_2| < 2 |\Gamma_1|$.

Furthermore, after an arbitrarily small perturbation of the metric, and in fact for any bumpy metric with strictly positive Ricci curvature, there exists a third solution $\Gamma_3$.
\end{theorem}

\begin{proof}
Recall that by the work of Meeks-Simon-Yau \cite{MSY} our manifold is diffeomorphic to the 3-ball. Hence, by Proposition \ref{optimalfoliation} (optimal foliation) it admits an optimal foliation $\{\Sigma_t\}_{t\in[-1,1]}$. In particular, the middle slice $\Gamma_1:=\Sigma_0$ is an embedded free-boundary minimal disk with area $\omega_1$, which is the least area among all solutions.

Next, by Proposition \ref{prop_flip_sweep} (flipping 2-sweepout) there exists a flipping 2-sweepout $\{\Sigma_{s,t}\}_{(s,t)\in [-1,1]^2}$ associated to our optimal foliation, such that
\begin{equation}\label{lessthan2}
\sup_{s,t\in [-1,1]^2} |\Sigma_{s,t}| < 2 |\Sigma_0|.
\end{equation}
Denote by $\Pi_2$ the saturation of this 2-sweepout, and consider its min-max width $\omega_2$. By item \eqref{def_flip1} of Definition \ref{def_flip} (flipping 2-sweepout) we clearly have $\omega_1\leq \omega_2$.
If $\omega_1=\omega_2$ then by Lemma \ref{lemma_eq} (equality case) below there exist infinitely many solutions of area $\omega_1$.

Suppose now $\omega_1 < \omega_2$. Also, note that thanks to \eqref{lessthan2} we have $\omega_2 < 2\omega_1$.
Hence, applying Theorem \ref{gj} (min-max theorem) for our 2-parameter family we get an embedded free-boundary minimal disk $\Gamma_2$ with
\begin{equation}
|\Gamma_1| < |\Gamma_2| < 2 |\Gamma_2|.
\end{equation}
Finally, the assertion about the third solution follows from \cite{MNS,MNS2}.
\end{proof}

Here, we used the following Lusternik-Schnirelman type lemma:

\begin{lemma}[equality case; c.f. \cite{HaslhoferKetover,Ketover_flip}]\label{lemma_eq}
If $\omega_1=\omega_2$, then $M$ contains infinitely many embedded free-boundary minimal disks of area $\omega_1$.
\end{lemma}

\begin{proof}Suppose towards a contradiction that the collection $\mathcal{S}$ of oriented embedded free-boundary minimal disks of area $\omega_1$ is finite. Choose a sequence $\Sigma^i_{s,t} \in \Pi_2$ be such that
\begin{equation}
\sup_{s,t\in [-1,1]^2} |\Sigma^i_{s,t}| < \omega_1 + \delta_i,
\end{equation}
where $\delta_i\to 0$. For any $\eps>0$, using the $\bf{F}$-metric on the space of integral varifolds, we consider the sets
\begin{equation}
A^i_\eps:= \left\{ (s,t)\in [-1,1]^2\, : \, {\bf{F}}(\Sigma^i_{s,t},\mathcal{S}) < \eps \right\}.
\end{equation}
We claim that for $i\geq i_0(\eps)$ there exists a continuous path in $A^i_\eps$ that connects the left and right boundary of the rectangle $[-1,1]^2$.

Indeed, if this failed for some $\eps>0$, then along a subsequence we could find continuous paths $\gamma^i(\tau)$ connecting the bottom and top boundary of the rectangle $[-1,1]^2$, such that
${\bf{F}}( \Sigma^i_{\gamma^i(\tau)}, \mathcal{S})\geq \eps$ for all $\tau$.
But then using the ``pull-tight" argument of Almgren and Pitts for $i$ large enough we would obtain
$\sup_\tau|\Sigma^i_{\gamma^i(\tau)}| < \omega_1$. This
contradicts the definition of a flipping 2-sweepout, and thus proves the claim.

It follows that for $i\geq \tilde{i}_0(\eps)$ we can find a continuous path $\gamma^i:[-1,1]\to A^i_\eps$ that connects $(-1,0)$ and $(1,0)$. However, since by definition of a flipping 2-sweepout the free-boundary minimal disks $\Sigma^i_{-1,0}$ and $\Sigma^i_{1,0}$ have the opposite orientation, for $\eps$ sufficiently small there cannot be any such path (since $\mathcal{S}$ is finite by assumption).  This gives the desired contradiction, and thus proves the lemma.
\end{proof}

\begin{remark}[families of solutions]
By a result of Nitsche \cite{Nitsche} all solutions in the round 3-ball are planar, and thus the family of solutions in the round 3-ball is parametrized by $\mathbb{RP}^2$. It is an interesting question for which deformations of the round 3-ball there exists an $\mathbb{RP}^2$-family of solutions (for related work on Zoll families, see e.g. \cite{Zoll,Guillemin,AMS}).
\end{remark}
 
Finally, let us explain how Corollary \ref{cor_ell} (ellipsoids) follows:

\begin{proof}[{Proof of Corollary \ref{cor_ell}}] As explained in the introduction, the first part of the corollary follows from Theorem \ref{thm_main} (minimal disks) by observing that for $a\geq 2\max (b,c)$ the areas of $xy$-planar disk and $xz$-planar disk are at least twice as large as the one of $yz$-planar disk. 

Now, fixing $b$ and $c$, recall that $E(a,b,c)$ for $a\to \infty$ converges to $\mathbb{R}\times E(b,c)$.  Observe that the nonplanar embedded free-boundary disks $\Sigma(a)$ provided by Theorem \ref{thm_main} (minimal disks) have area less than $2\pi b c$, and intersect $\{x=0\}$ by the Frankel property, and hence must be contained in $\{ |x|\leq C\}$ by the monotonicity formula.

For  $a\to \infty$ our surfaces $\Sigma(a)$ subsequentially converges to a free-boundary integral varifold $V$ in $E(b,c)\times\mathbb{R}$. By the above, the mass of $V$ is bounded by $2\pi b c$, and the support of $V$ intersects $\{x=0\}$, and is contained in the slab $\{ |x|\leq C\}$. Hence, sliding the flat disks $\{x=t\}\cap E(b,c)$ until they touch $V$, we infer that $V$ is equal to the central disk $\{x=0\}\cap E(b,c)$ with multiplicity $k=1$ or $2$.

Suppose towards a contradiction that $k=1$.  Then by the local regularity theorem the convergence is smooth.  Since $\Sigma(a)$ intersects $\{x=0\}\cap E(a,b,c)$, we thus obtain a nonzero Jacobi (i.e. harmonic) function $u$ on $\{0\} \times E(b,c)$ that has at least one zero.  Moreover, $u$ has zero Neumann boundary data, since $h(\nu,\nu)\to 0$ as $a\to\infty$. Thus $u$ vanishes identically, a contradiction. This completes the proof.
\end{proof}

\begin{remark}[ellipsoids]
It is an interesting question to determine the set of parameters $a\geq b\geq c$ for which nonplanar solutions exist.
\end{remark}

\section{Existence of smooth mean-convex foliations}\label{sec_foliationsl}

Throughout this section, $M$ denotes a smooth 3-ball equipped with an arbitrary Riemannian metric with strictly convex boundary. As usual we say that a smooth domain $K\subset M$ has \emph{free-boundary} if $\partial K:=K\setminus \mathrm{Int}_M(K)$ meets $\partial M$ orthogonally.\\

The goal of this section is to prove Theorem \ref{thm_fol} (free-boundary foliation dichotomy). To do so, we will adopt the approach from \cite{BHH,HaslhoferKetover} to the free-boundary setting. Specifically, we consider the free-boundary flow with surgery from \cite{Haslhofer_fb_surgery}, but modify it so that necks and half necks are deformed to tiny strings and half strings instead of being cut out completely. Using this we will show that there exists a perfect isotopy $K'_t$ that deforms $K$ into a so-called half marble tree (once this is achieved, the family $\Sigma_t=\partial K'_t$ concatenated with a suitable final bit that shrinks the half marble tree to a boundary point will give the desired foliation of $K$). Here and in the following, a \emph{perfect isotopy} is a smooth isotopy by strictly mean-convex domains $K'_t$ satisfying the free-boundary condition that is strictly monotone, i.e. 
\begin{equation}
K'_{t_2}\subseteq \mathrm{Int}_M K'_{t_1}\quad \mathrm{for} \,\, t_2>t_1.
\end{equation}
A half marble tree is, roughly speaking, a connected sum of round balls and half balls along tiny strings and half strings. The precise definition is as follows:

\begin{definition}[half marble tree]\label{def_marbletree}
A \emph{half marble tree} in $M$ with string radius $r_s$ and marble radius $r_m$ is a domain of the form $\mathcal{G}^M_{r_s}(B,\gamma)\subset M$, where $\mathcal{G}^M_{r_s}$ denotes the gluing map specified below, such that
\begin{enumerate}
\item $B=\bigcup_i B_{r_m}(p_i)\cup \bigcup_{j} B^+_{r_m}(q_j)$ is a disjoint union of finitely many closed round balls (marbles) and half marbles of radius $r_m$ centered at $p_i\in M$ with $d(p_i,\partial M)\geq 5r_m$ and at $q_j\in \partial M$, respectively.
\item The curve $\gamma=\bigcup_i \zeta_i \cup \bigcup_{j} \xi_j$ is a disjoint union of finitely many connected smooth compact embedded curves, satisfying $d(\zeta_i,\partial M)\geq r_m/2$ and $\xi_i\subset\partial M$, respectively, such that
\begin{itemize}
\item $\gamma$ is disjoint from $\textrm{Int}(B)$ and satisfies $\partial\gamma\subset \partial B$,
\item $B\cup \gamma$ is simply connected,
\item $\zeta_i \cap B_{3r_m}(p_i)$ is a union of radial geodesics and $\xi_i \cap B^+_{3r_m}(q_i)$ is a union of straight boundary lines, respectively.
\end{itemize}
\end{enumerate}
\end{definition}

To elaborate on this definition, recall that for any $q\in \partial M$ the normal exponential map defines a diffeomorphism $\Phi_q$ from an open neighborhood of $q$ in $M$ to an open neighborhood of the origin in the upper half space $\mathbb{R}^3_+:=\mathbb{R}^2\times \mathbb{R}_{\geq 0}$. Using this, a half marble $B^+_{r_m}(q)$ is defined as the inverse image under $\Phi_q$ of a closed half ball of radius $r_m$ in $\mathbb{R}^3_+$ centered at the origin, and a straight boundary line is defined as inverse image under $\Phi_q$ of a straight line in $\mathbb{R}^2\times \{0\}$ starting at the origin.\\

\subsection{Gluing map}
Generalizing the arguments from \cite{BHH}, we will now construct a suitable gluing map for $(\mathbb{A},b)$-controlled configurations of domains and curves. In our setting, an $\mathbb{A}=(\alpha,c_H,C_A,D_A)$-controlled domain is a free-boundary mean-convex domain $K\subset M$  that is $\alpha$-noncollapsed and satisfies $H\geq c_H$, $|A|\leq C_A$, $|\nabla A|\leq D_A$, and a $b$-controlled curve is a compact embedded curve $\gamma=\zeta \cup \xi\subset M$, where $d(\zeta,\partial M)\geq 10b$ and $\xi\subset\partial M$, respectively, such that each connected component has length at least $10 b$ and normal injectivity radius at least $\frac{1}{10}b$, such that different connected components are distance at least $10 b$ apart, and such that the curvature satisfies $|\kappa|\leq b^{-1}$ and $|\partial_s \kappa|\leq b^{-2}$.

\begin{definition}[controlled configuration]\label{conrolled_conf}
An $(\mathbb{A},b)$-controlled configuration of domains and curves in $M$ is a pair $(K,\gamma)$, where $K\subset M$ is an $\mathbb{A}$-controlled domain and $\gamma\subset M$ is a $b$-controlled curve, such that
\begin{enumerate}
\item The interior of $\gamma$ lies entirely in $M\setminus K$.
\item The endpoints of $\gamma$ satisfy the following properties:
\begin{itemize}
\item If $p\in \partial\gamma\cap\partial K$, then $\gamma$ touches $\partial K$ orthogonally there.
\item If $p\in \partial\gamma\setminus\partial K$, then $d(p,\partial K)\geq 10b$.
\item $d \big( \gamma \setminus\bigcup_{p\in \partial\gamma}B_{b/10}(p),\partial K \big)\geq b/20$.
\end{itemize}
\end{enumerate}
\end{definition}

Denote by $\mathcal{X}_{\mathbb{A},b}$ the set of all $(\mathbb{A},b)$-controlled configurations of domains and curves in $M$, and by $\mathcal{D}$ the set of all free-boundary strictly mean-convex domains in $M$. A map $\mathcal{F}: \mathcal{X}_{\mathbb{A},b}\to \mathcal{D}$ is called smooth, if $\mathcal{F}\circ \phi$ is smooth for every smooth finite parameter family $\phi:B^k\to \mathcal{X}_{\mathbb{A},b}$.

\begin{proposition}[{gluing map, c.f. \cite[Theorem 4.1]{BHH}}]\label{prop_gluing_map}
Given any smooth compact strictly convex domain $M$, there exists a smooth map
\begin{equation}
\mathcal{G}^M:\mathcal{X}_{\mathbb{A},b}\times (0,\bar{r})\rightarrow \mathcal{D},\quad ((K,\gamma),r)\mapsto \mathcal{G}^M_{r}(K,\gamma),
\end{equation}
where $\bar{r}=\bar{r}(M,\mathbb{A},b)>0$ is a constant, and an increasing function $\rho:(0,\bar{r})\rightarrow \mathbb{R}_+$ with $\lim_{r\rightarrow 0}\rho(r)=0$, with the following significance:
\begin{enumerate}
\item\label{g1} $\mathcal{G}^M_{r}(K,\gamma)$ deformation retracts to $K\cup \gamma$.
\item\label{g2} Writing $\gamma=\zeta \cup \xi\subset M$, where $\zeta\subset \mathrm{Int}(M)$ and $\xi\subset\partial M$, we have
\begin{align*}
&\mathcal{G}^M_{r}(K,\gamma)  \triangle\! \left(K \cup  N_{r}(\zeta) \cup  N^+_{r}(\xi) \right)  \\
&\qquad\qquad\qquad \subseteq \bigcup_{p\in \partial\zeta} B_{\rho(r)}(p) \cup \bigcup_{q\in \partial\xi} B^+_{\rho(r)}(q),
\end{align*}
where $N_{r}(\zeta)$ and $N_r^+(\xi)$ denote the $r$-tubular neighborhood of $\zeta$ and half $r$-tubular neighborhood of $\xi$, respectively.
\item\label{g3}  If $p\in \partial \zeta\setminus\partial K$, then we have
\begin{align*}
\mathcal{G}^M_{r}(K,\gamma)\cap B_{\rho(r)}(p)= 
CN_{r}(\zeta)\cap B_{\rho(r)}(p),  
\end{align*}
and if $q\in \partial \xi\setminus\partial K$, then we have
\begin{align*}
\mathcal{G}^M_{r}(K,\gamma)\cap B^+_{\rho(r)}(q)= 
CN^+_{r}(\xi)\cap B^+_{\rho(r)}(q),  
\end{align*}
where $CN_{r}(\zeta)$ and $CN^+_{r}(\xi)$ denote the capped-off $r$-tube around $\zeta$ and capped-off half $r$-tube around $\xi$, respectively.
\end{enumerate}
\end{proposition}

Here, we recall from \cite[Section 2]{BHH} that fixing a suitable standard cap $K^{\mathrm{st}}\subset\mathbb{R}^3$ the capped-off $r$-tube $CN_{r}(\zeta)$ is defined by putting standard caps at the endpoints $p\in \partial \zeta$ via the normal exponential map. Similarly, we can define $CN_{r}^+(\xi)$ by putting standard half caps $K^{\mathrm{st}}\cap \mathbb{R}^3_+$ at the endpoints $q\in \partial \xi$ via the boundary straightening map $\Phi$.

\begin{proof}Choosing the function $\rho(r)$ similarly as in \cite[Theorem 4.1]{BHH}, in light of \eqref{g2} and \eqref{g3} it suffices to specify the gluing map in $B_{\rho(r)}(p)$ for $p\in\partial\zeta\cap \partial K$ and in $B^+_{\rho(r)}(q)$ for  $q\in\partial\xi\cap \partial K$.
 
 To do the former, we can simply map $B_{\rho(r)}(p)$ to a ball in $\mathbb{R}^3$ via a suitable normal exponential map, apply the gluing map in $\mathbb{R}^3$ from \cite[Theorem 4.1]{BHH}, and map back to $M$.  Here, to be precise one has to extend the image of $( K, \gamma)\cap B_{\rho(r)}(p)$ to a controlled configuration defined in entire $\mathbb{R}^3$, but by the locality property of the cited theorem this is well-defined, i.e. independent of the choice of extension.

 To do the latter, using the boundary straightening map $\Phi$ it suffices to construct a gluing map with the desired properties in $\mathbb{R}^3_+$. Specifically, we observe that the proof of \cite[Theorem 4.1]{BHH} also applies to the half space setting. Namely, given $q\in\partial\xi\cap \partial K\subset \mathbb{R}^3_+$, arguing as in the proof of \cite[Proposition 4.16]{BHH} we can deform $K$ to its round second order approximation in $B^+_{\rho(r)/2}(q)$ without changing $K\setminus B^+_{\rho(r)}(q)$. Next, arguing as in the proof of \cite[Proposition 4.11]{BHH} we can transition to the rotationally symmetric setting in  $B^+_{c\rho(r)/2}(q)$, where $c=c(\mathbb{A})>0$ is a small constant. Finally, using the explicit model from  \cite[Proposition 4.2]{BHH} we can glue the round half ball and the round half cylinder.
\end{proof}

\subsection{Isotopies through surgeries}
For the purpose of the present paper a free-boundary flow with surgery is a free-boundary $(\delta,\mathcal{H})$-flow in the domain $M$ as defined in \cite[Definition 2.4]{Haslhofer_fb_surgery}. In particular, we recall that $\delta>0$ is a small parameter that captures the quality of the surgery necks and half necks, and $\mathcal{H}$ is a triple of curvature scales $H_{\textrm{trigger}}\gg H_{\textrm{neck}}\gg H_{\textrm{thick}}\gg 1$, which is used to specify more precisely when and how surgeries are performed. We also recall that there is a discrete set of surgery times $t_i$, where
\begin{itemize}
\item some necks or half necks in the presurgery domain $K_{t_i}^-$ are replaced by caps or half caps yielding a domain $K_{t_i}^\sharp\subseteq K_{t_i}^-$, 
\item and/or some connected components are discarded yielding the postsurgery domain $K_{t_i}^+\subseteq K_{t_i}^\sharp$.
\end{itemize}

Let us first deal with the discarded components. To this end, recall that an $\eps$-tube $(K,\gamma)$ is a compact mean-convex domain $K\subset\mathrm{Int}(M)$ homeomorphic to a ball, together with a connected curve $\gamma\subset K$ with endpoints in $\partial K$, satisfying the properties specified in \cite[Definition 7.3]{BHH}.
Similarly, a half $\eps$-tube $(K,\gamma)$ is a compact mean-convex free-boundary domain $K\subset M$ homeomorphic to a ball, together with a connected curve $\gamma\subset K\cap \partial M$ with endpoints $\bar{q}_\pm\in \partial K$, such that for some $C<\infty$ we have the following: (i) $K\cap B^+_{2CH^{-1}(\bar{q}_\pm)}(\bar{q}_\pm)$ is, after rescaling by $H(\bar{q}_\pm)$, $\eps$-close in $C^{\lfloor 1/\eps\rfloor}$ to a standard half cap or the half bowl soliton, and (ii)
every $q\in \gamma$ with $d(q,\bar{q}_\pm)\geq CH^{-1}(\bar{q}_\pm)$ is the center of a half $\eps$-neck with axis given by $\partial_s\gamma(q)$, and $\gamma$ is  $\eps^{-2}r$-controlled in $B^+_{\eps^{-1}r}(q)$, where $r$ denotes the radius of the half $\eps$-neck.

\begin{proposition}[isotopy for capped tubes and half tubes]\label{prop_iso_tubes}
For $H_{\textrm{neck}}<\infty$ large enough, and $\eps>0$ small enough, for every capped $\eps$-tube and capped half $\eps$-tube there exists a perfect isotopy to a half marble tree.
\end{proposition}

\begin{proof}
Let us first consider the assertion in $\mathbb{R}^3$ and the upper half space $\mathbb{R}^3_+$, respectively. The former has already been established in \cite[Proposition 6.3]{HaslhoferKetover}.

To establish the assertion in $\mathbb{R}^3_+$, we will adopt the proof of the cited proposition to the setting of half tubes. Given a half $\eps$-tube $(K,\gamma)$ in $\mathbb{R}^3_+$, where $\varepsilon>0$ is sufficiently small, let $q_\pm\in\gamma$ be half $\varepsilon$-neck points that are as close as possible to the endpoints $\bar{q}_\pm$, respectively. Let $\mathcal{I}\subset \gamma$ be a maximal collection of half $\varepsilon$-neck points including $q_{-}$ such that for any pair $q_1,q_2\in \mathcal{I}$ the separation between the points is at least $100 \varepsilon^{-1} \max\{H(q_1)^{-1}, H(q_2)^{-1}\}$.
For each $q\in\mathcal{I}$ we replace the half $\varepsilon$-neck with center $q$ by a pair of opposing standard half caps with suitable separation parameter $\Gamma<\infty$ as in \cite[Definition 2.3]{Haslhofer_fb_surgery}. Denote the postsurgery domain by $K^\sharp$, and let $\tilde{\gamma}$ be a disjoint union of almost straight lines connecting the opposing standard half caps.

Let $\mathcal{G}_{r_s}$ be the gluing map in $\mathbb{R}^3_+$ with small enough string radius $r_s$. Let $K^\sharp_t$ be the free-boundary mean curvature flow evolution of $K^\sharp$, and let $\tilde{\gamma}_t$ be the family of curves which follows $K^\sharp_t$ by normal motion starting at $\tilde{\gamma}_0=\tilde{\gamma}$. We claim that for $\bar{t}$ small enough, and for a suitable family of curves $\tilde{\gamma}'_t\subset\partial \mathbb{R}^3_{+}$ very close to $\tilde{\gamma}_t$, there exists a perfect isotopy between $K$ and $\mathcal{G}_{r_s}(K^\sharp_{\bar{t}},\tilde{\gamma}'_{\bar{t}})$.

To see this we fix a partition $0<\bar{t}_1<\bar{t}_2<\bar{t}_3<\bar{t}$, where $\bar{t}$ is small enough, and construct a perfect isotopy step by step as follows.

First, if there is a half surgery with center $q$, then for $t$ small enough $K^\sharp_t$ can be expressed locally as a free-boundary graph with small $C^{20}$-norm over a pair of opposing evolving standard half caps. Hence, we can find a perfect isotopy $\{L'_t\}_{t\in [0,\bar{t}_1]}$ starting at $L'_0=K^\sharp$, such that for each half surgery center $q$ we have that $L'_{\bar{t}_1}\cap B_{5\Gamma H_{\textrm{neck}}^{-1}}(q)$ is exactly a pair of opposing standard half caps. Moreover, we can slightly perturb the family $\tilde{\gamma}_t$ to get a family $\tilde{\gamma}'_t\subset \partial \mathbb{R}^3_{+}$ with the property that 
$\tilde{\gamma}'_{\bar{t}_1}$ connects these opposing standard caps in exactly straight lines.

Next, using the above, similarly as in \cite[Proposition 3.12]{BHH} we can find a perfect isotopy $\{L_t\}_{t\in [0,\bar{t}_2]}$ starting at $K$, such that at time $\bar{t}_2$ for each half surgery point $q$ we have that $L_{\bar{t}_2}\cap B_{5\Gamma H_{\textrm{neck}}^{-1}}(q)$ is pair of standard half caps connected by a round half neck of radius $\varrho(0.98)H_{\textrm{neck}}^{-1}$, where $\varrho$ is the radius function from gluing the ball and the cylinder from \cite[Proposition 4.2]{BHH}.

Third, using the fact that the gluing in $\mathbb{R}^3_+$ in the rotationally symmetric case is described by an explicit model, we can now decrease the neck radius from $\varrho(0.98) H_{\textrm{neck}}^{-1}$ down to $2r_s$ via a perfect isotopy $\{L_t\}_{t\in [\bar{t}_2,\bar{t}_3]}$.

Finally, interpolating again between the rotationally symmetric and nonsymmetric situation via the graphical representation as above we can find a perfect isotopy $\{L_t\}_{t\in [\bar{t}_3,\bar{t}]}$ with $L_{\bar{t}}= \mathcal{G}_{r_s}(K^\sharp_{\bar{t}},\tilde{\gamma}'_{\bar{t}})$.

It remains to construct a perfect isotopy that deforms $\mathcal{G}_{r_s}(K^\sharp_{\bar{t}},\tilde{\gamma}'_{\bar{t}})$ into a half marble tree. To this end, if we choose $\bar{t}$ very small, then $K^\sharp_{\bar{t}}$ is as close as we want to $K^\sharp$. Then, inferring as in the proof of \cite[Proposition 7.4]{BHH} that the connected components of $K^\sharp_{\bar{t}}$ are either convex or capped off half cylinders, we see that there exists a perfect isotopy $\{L_t\}_{t\in [0,1]}$ starting at $L_0=K^\sharp_{\bar{t}}$ such that $L_1$ is a finite union of round half balls. 
Denoting by $r_{\textrm{min}}$ the smallest among the radii of the half balls of $L_1$, let $\{L_t\}_{t\in [1,2]}$ be a perfect isotopy that concatenates smoothly at $t = 1$ and shrinks all half balls further to half balls of radius $r_{\textrm{min}}/2$.
Let $\{\gamma_t\}_{t\in[0,2]}$ be the family of curves that follows $L_t$ by normal motion starting at $\gamma_0=\tilde{\gamma}'_{\bar{t}}$, and fix a slowly decreasing smooth positive function $r_s(t)$ starting at $r_s(0)=r_s$ from above.
Then, $\{\mathcal{G}_{r_s(t)}(L_t,\gamma_t)\}_{t\in [0,2]}$ is a perfect isotopy that deforms the domain $\mathcal{G}_{r_s}(K^\sharp_{\bar{t}},\tilde{\gamma}'_{\bar{t}})$ into a half marble tree with marble radius $r_{\textrm{min}}/2$ and string radius $r_s(2)$.

Taking a smooth concatenation of the above isotopies, establishes the assertion in $\mathbb{R}^3_{+}$.

Finally, for general $M$, choosing $H_{\textrm{neck}}$ very large the exponential map $\exp_p$ and the boundary straightening map $\Phi_q$ are locally as close as we want to the identity map. Hence, the above argument applies.
\end{proof}

\begin{corollary}[isotopy through surgeries]\label{cor_iso_trough_surg}
For $H_{\textrm{neck}}<\infty$ large enough, and $\delta>0$ small enough the following holds. Suppose $K^\sharp$ is obtained from $K^-$ by performing surgeries on a disjoint collection of $\delta$-necks and half $\delta$-necks, and let $\gamma$ be the union of almost straight lines connecting the tips of the opposing standard caps. Let $\{K^\sharp_t\}$ be a perfect evolution of $K^\sharp$, and let $\{\gamma_t\}$ be the family of curves which follows $K^\sharp_t$ by normal motion starting at $\gamma$.
Then, for $r_s$ small enough, every small enough $\bar{t}$, and a suitable perturbation of $\{\gamma_t\}$, which we denote again by $\{\gamma_t\}$, there exists a perfect isotopy between $K^-$ and $\mathcal{G}^M_{r_s}(K^\sharp_{\bar{t}},\gamma_{\bar{t}})$.
\end{corollary}
\begin{proof}
As before, choosing $H_{\textrm{neck}}<\infty$ large enough we can reduce to proving the corresponding statements in $\mathbb{R}^3$ and the upper half space $\mathbb{R}^3_+$, respectively. The former has already been been established in \cite[Corollary 6.4]{HaslhoferKetover}, and the latter follows by inspecting the above proof.
\end{proof}

\subsection{Proof of the foliation theorem}
After the above preparations, we can now prove the main result of this section.

\begin{proof}[{Proof of Theorem \ref{thm_fol} (free-boundary foliation dichotomy)}]
Given $K\subset M$ as in the statement of the theorem, we consider its evolution $\{K_t\}_{t\geq 0}$ by free-boundary flow with surgery with initial condition $K$, where we choose the surgery parameters suitably so that both the existence theorem and the canonical neighborhood theorem from \cite{Haslhofer_fb_surgery} apply. Moreover, we can arrange that at scale $H_{\textrm{neck}}^{-1}$ the ambient space looks as close as we want to Euclidean space. Furthermore, fixing $r_s\ll r_m\ll H_{\textrm{neck}}^{-1}$ we can ensure that the gluing map $\mathcal{G}^M_{r_s}$ from Proposition \ref{prop_gluing_map} (gluing map) is well defined.

By \cite[Theorem 1.1]{Haslhofer_fb_surgery} the free-boundary flow with surgery either becomes extinct at some $T<\infty$ or for $t\to \infty$ converges smoothly in the one or two-sheeted sense to a finite collection of embedded stable connected minimal surfaces with free or empty boundary. Since we started with a disk, by the nature of the surgery process \cite[Definition 2.4]{Haslhofer_fb_surgery}, these must be either embedded free-boundary minimal disks in $\mathrm{Int}_M(K)$ or embedded minimal two-spheres in $\mathrm{Int}_M(K)\setminus \partial M$. Suppose from now on the free-boundary flow with surgery becomes extinct at some $T<\infty$. Consider the times $0 < t_1 < \ldots < t_\ell\leq T$ when there is some surgery and/or discarding. 

\begin{claim}[discarded components]\label{claim_inductionhyp}
For all discarded components $C_i^j$ at time $t_i$ there exists a perfect isotopy to a half marble tree.
\end{claim}

\begin{proof}[{Proof of Claim \ref{claim_inductionhyp}}]
Our topological assumption on the initial domain $K$ together with the nature of the surgery process implies that all discarded components are homeomorphic to balls. Thus, by the canonical neighborhood theorem from \cite{Haslhofer_fb_surgery}, each discarded component is either (a) convex with controlled geometry or (b) a capped $\eps$-tube or capped half $\eps$-tube (to be precise, there is also the potential scenario of a tube that is capped at one end and has free-boundary on the other end, but this is dealt with almost exactly the same way as a capped $\eps$-tube). Contracting to a small ball or half ball, respectively, in case (a), and using Proposition \ref{prop_iso_tubes} (isotopy for capped tubes and half tubes) in case (b), we can find a perfect isotopy to a half marble tree.
\end{proof}

Continuing the proof of the theorem, let $\A_i$ be the assertion that for each connected component of the presurgery domain $K^i:=K_{t_{i}}^-$ there exist a perfect isotopy to a half marble tree. Since $K_{t_{\ell}}^+=\emptyset$, we see that at the final time $t_{\ell}$ there is no replacement of necks or half necks by caps or half caps. Thus, all connected components of $K^\ell=K_{t_\ell}^-$ get discarded, and Claim \ref{claim_inductionhyp} (discarded components) shows that $\A_\ell$ holds.

\begin{claim}[induction step]\label{claim_inductionstep}
If $0<i<\ell$ and $\A_{i+1}$ holds, so does $\A_i$.
\end{claim}

\begin{proof}[{Proof of Claim \ref{claim_inductionstep}}]
Smooth evolution by mean curvature flow with free-boundary provides a perfect isotopy between $K_{t_i}^+$ and $K^{i+1}$. To proceed, recall that $K_{t_{i}}^+\subseteq K_{t_{i}}^\sharp\subseteq K_{t_{i}}^-=K^i$ is obtained by performing surgery on a minimal collection of disjoint $\delta$-necks and half $\delta$-necks separating the thick part and the trigger part and/or discarding connected components that are entirely covered by canonical neighborhoods.

By induction hypothesis each connected component of $K^{i+1}$ is perfectly isotopic to a half marble tree, and by Claim \ref{claim_inductionhyp} (isotopy for discarded components) each discarded component  is perfectly isotopic to a half marble tree as well. It follows that there exists a perfect isotopy $\{L_t\}_{t\in [0,1]}$ deforming $L_0=K_{t_{i}}^\sharp$ into a union of half marble trees $L_1$. If $L_0$ has more than one connected component, then we glue together these perfect isotopies using Proposition \ref{prop_gluing_map} (gluing map) as follows.

For each surgery neck and half neck at time $t_i$, select an almost straight line $\gamma_i^j$ between the tips of the corresponding pair of standard caps and half caps in $K_{t_{i}}^\sharp$. Then, setting $\gamma:=\bigcup_j \gamma_i^j$, by Corollary \ref{cor_iso_trough_surg} (isotopy trough surgeries) the domain $K^i=K_{t_i}^-$ is perfectly isotopic to $\mathcal{G}^M_{r_s}(L_{\bar{t}}^\sharp,\gamma_{\bar{t}})$, provided $\bar{t}$ is small enough. Finally, define $\{\gamma_t\}_{t\in [\bar{t},1]}$ by following the points where $\gamma_t$ touches $\partial L_t$ via normal motion, with the following modification. It can happen at finitely many times $t$ that $\gamma_{t}$ hits $\partial B_r(p)$ or $\partial B^+_r(q)$, for some surgery center $p$ or half surgery center $q$, where $r=10\Gamma H_{\textrm{neck}}^{-1}$. Whenever this happens, we modify $\gamma_t$ according to \cite[Lemma 9.4]{BHH} by sliding along the neck or half neck. Choose a suitable slowly decreasing positive function  $r_s(t)$ starting at $r_s(0)=r_s$ from above. Then $\mathcal{G}^M_{r_s(t)}(L_t,\gamma_t)_{t\in [\bar{t},1]}$ gives the last bit of the desired perfect isotopy. This finishes the proof of Claim \ref{claim_inductionstep}.
\end{proof}

To conclude the proof of the theorem, using Claim \ref{claim_inductionstep} (induction step) it follows from backwards induction on $i$ that $\A_1$ holds. Moreover, smooth mean curvature flow with free-boundary provides a perfect isotopy between $K$ and $K^1$. In particular, $K^1$ has only one connected component, and $\partial K^1\cap \partial M\neq\emptyset$. Thus, there exists a perfect isotopy deforming $K$ into a half marble tree, which has at least one half marble. Finally, arguing as in \cite[Section 5]{BHH} we can shrink the half marble tree in a perfect way to a boundary point to obtain the final bit of the desired foliation. This finishes the proof of Theorem \ref{thm_fol}.
\end{proof}

\begin{remark}[mean-convex 3-balls]
It seems likely that the conclusion of Theorem \ref{thm_fol} (free-boundary foliation dichotomy) still holds under the weaker assumption that $\partial M$ is strictly mean-convex. In fact, the only place where strict convexity was used is the proof of \cite[Claim 3.2]{Haslhofer_fb_surgery}.
\end{remark}

\bigskip

\bibliography{HaslhoferKetover_fb}

\begin{thebibliography}{10}

\bibitem{MNS2}
A.~Ach\'{e}, D.~Maximo, and H.~Wu.
\newblock Metrics with nonnegative {R}icci curvature on convex three-manifolds.
\newblock {\em Geom. Topol.}, 20(5):2905--2922, 2016.

\bibitem{ABN}
S.~Alexakis, T.~Balehowsky, and A.~Nachman.
\newblock Determining a {R}iemannian metric from minimal areas.
\newblock {\em Adv. Math.}, 366:107025, 71, 2020.

\bibitem{AMS}
L.~Ambrozio, F.~Marques, and A.~Neves.
\newblock Riemannian metrics on the sphere with {Z}oll families of minimal
  hypersurfaces.
\newblock {\em arXiv:2112.01448}, 2021.

\bibitem{Birkhoff}
G.~Birkhoff.
\newblock Dynamical systems with two degrees of freedom.
\newblock {\em Trans. Amer. Math. Soc.}, 18(2):199--300, 1917.

\bibitem{BHH}
R.~Buzano, R.~Haslhofer, and O.~Hershkovits.
\newblock The moduli space of two-convex embedded spheres.
\newblock {\em J. Differential Geom.}, 118(2):189--221, 2021.

\bibitem{ChodoshMantoulidis}
O.~Chodosh and C.~Mantoulidis.
\newblock Minimal surfaces and the {A}llen-{C}ahn equation on 3-manifolds:
  index, multiplicity, and curvature estimates.
\newblock {\em Ann. of Math. (2)}, 191(1):213--328, 2020.

\bibitem{DMR}
F.~Da~Lio, L.~Martinazzi, and T.~Rivi\`ere.
\newblock Blow-up analysis of a nonlocal {L}iouville-type equation.
\newblock {\em Anal. PDE}, 8(7):1757--1805, 2015.

\bibitem{DLR}
C.~De~Lellis and J.~Ramic.
\newblock Min-max theory for minimal hypersurfaces with boundary.
\newblock {\em Ann. Inst. Fourier (Grenoble)}, 68(5):1909--1986, 2018.

\bibitem{DHKW}
U.~Dierkes, S.~Hildebrandt, A.~K\"{u}ster, and O.~Wohlrab.
\newblock {\em Minimal surfaces. {I}}, volume 295 of {\em Grundlehren der
  mathematischen Wissenschaften [Fundamental Principles of Mathematical
  Sciences]}.
\newblock Springer-Verlag, Berlin, 1992.
\newblock Boundary value problems.

\bibitem{EHIZ}
N.~Edelen, R.~Haslhofer, M.~Ivaki, and J.~Zhu.
\newblock Mean convex mean curvature flow with free boundary.
\newblock {\em Comm. Pure Appl. Math.}, 75(4):767--817, 2022.

\bibitem{Franz}
G.~Franz.
\newblock Equivariant index bound for min-max free boundary minimal surfaces.
\newblock {\em arXiv:2110.01020}, 2021.

\bibitem{Fraser_fb}
A.~Fraser.
\newblock On the free boundary variational problem for minimal disks.
\newblock {\em Comm. Pure Appl. Math.}, 53(8):931--971, 2000.

\bibitem{FraserLi}
A.~Fraser and M.~Li.
\newblock Compactness of the space of embedded minimal surfaces with free
  boundary in three-manifolds with nonnegative {R}icci curvature and convex
  boundary.
\newblock {\em J. Differential Geom.}, 96(2):183--200, 2014.

\bibitem{Grayson}
M.~Grayson.
\newblock Shortening embedded curves.
\newblock {\em Ann. of Math. (2)}, 129(1):71--111, 1989.

\bibitem{GruterJost}
M.~Gr\"{u}ter and J.~Jost.
\newblock On embedded minimal disks in convex bodies.
\newblock {\em Ann. Inst. H. Poincar\'{e} Anal. Non Lin\'{e}aire},
  3(5):345--390, 1986.

\bibitem{GMN}
M.~Guaraco, F.~Marques, and A.~Neves.
\newblock Multiplicity one and strictly stable {A}llen-{C}ahn minimal
  hypersurfaces.
\newblock {\em arXiv:1912.08997}, 2019.

\bibitem{Guillemin}
V.~Guillemin.
\newblock The {R}adon transform on {Z}oll surfaces.
\newblock {\em Advances in Math.}, 22(1):85--119, 1976.

\bibitem{Haslhofer_fb_surgery}
R.~Haslhofer.
\newblock Free boundary flow with surgery.
\newblock {\em arXiv:2306.07714}, 2023.

\bibitem{HaslhoferKetover}
R.~Haslhofer and D.~Ketover.
\newblock Minimal 2-spheres in 3-spheres.
\newblock {\em Duke Math. J.}, 168(10):1929--1975, 2019.

\bibitem{Hatcher}
A.~Hatcher.
\newblock A proof of the {S}male conjecture, {${\rm Diff}(S^{3})\simeq {\rm
  O}(4)$}.
\newblock {\em Ann. of Math. (2)}, 117(3):553--607, 1983.

\bibitem{IMN}
K.~Irie, F.~Marques, and A.~Neves.
\newblock Density of minimal hypersurfaces for generic metrics.
\newblock {\em Ann. of Math. (2)}, 187(3):963--972, 2018.

\bibitem{Jost2}
J.~Jost.
\newblock Existence results for embedded minimal surfaces of controlled
  topological type. {II}.
\newblock {\em Ann. Scuola Norm. Sup. Pisa Cl. Sci. (4)}, 13(3):401--426, 1986.

\bibitem{J}
J.~Jost.
\newblock Embedded minimal surfaces in manifolds diffeomorphic to the
  three-dimensional ball or sphere.
\newblock {\em J. Differential Geom.}, 30:555--577, 1989.

\bibitem{Ketover_flip}
D.~Ketover.
\newblock Flipping {H}eegaard splittings and minimal surfaces.
\newblock {\em arXiv:2211.03745}, 2022.

\bibitem{KMN}
D.~Ketover, F.~Marques, and A.~Neves.
\newblock The catenoid estimate and its geometric applications.
\newblock {\em J. Differential Geom.}, 115(1):1--26, 2020.

\bibitem{LaurainPetrides}
P.~Laurain and R.~Petrides.
\newblock Existence of min-max free boundary disks realizing the width of a
  manifold.
\newblock {\em Adv. Math.}, 352:326--371, 2019.

\bibitem{Li}
M.~Li.
\newblock A general existence theorem for embedded minimal surfaces with free
  boundary.
\newblock {\em Comm. Pure Appl. Math.}, 68(2):286--331, 2015.

\bibitem{LiZhou}
M.~Li and X.~Zhou.
\newblock Min-max theory for free boundary minimal hypersurfaces
  {I}---{R}egularity theory.
\newblock {\em J. Differential Geom.}, 118(3):487--553, 2021.

\bibitem{LSZ}
L.~Lin, A.~Sun, and X.~Zhou.
\newblock Min-max minimal disks with free boundary in {R}iemannian manifolds.
\newblock {\em Geom. Topol.}, 24(1):471--532, 2020.

\bibitem{LMN}
Y.~Liokumovich, F.~Marques, and A.~Neves.
\newblock Weyl law for the volume spectrum.
\newblock {\em Ann. of Math. (2)}, 187(3):933--961, 2018.

\bibitem{LS}
L.~Lusternik and L.~Schnirelmann.
\newblock Topological methods in variational problems and their application to
  the differential geometry of surfaces.
\newblock {\em Uspehi Matem. Nauk (N.S.)}, 2(1(17)):166--217, 1947.

\bibitem{Marques_Ricci_moduli}
F.~Marques.
\newblock Deforming three-manifolds with positive scalar curvature.
\newblock {\em Ann. of Math. (2)}, 176(2):815--863, 2012.

\bibitem{MN_Willmore}
F.~Marques and A.~Neves.
\newblock Min-max theory and the {W}illmore conjecture.
\newblock {\em Ann. of Math. (2)}, 179(2):683--782, 2014.

\bibitem{MN_Morse1}
F.~Marques and A.~Neves.
\newblock Morse index and multiplicity of min-max minimal hypersurfaces.
\newblock {\em Camb. J. of Math.}, 4(4):463--511, 2016.

\bibitem{MN_inf}
F.~Marques and A.~Neves.
\newblock Existence of infinitely many minimal hypersurfaces in positive
  {R}icci curvature.
\newblock {\em Invent. Math.}, 209(2):577--616, 2017.

\bibitem{MN_Morse2}
F.~Marques and A.~Neves.
\newblock Morse index of multiplicity one min-max minimal hypersurfaces.
\newblock {\em Adv. Math.}, 378:Paper No. 107527, 58, 2021.

\bibitem{MNS_equi}
F.~Marques, A.~Neves, and A.~Song.
\newblock Equidistribution of minimal hypersurfaces for generic metrics.
\newblock {\em Invent. Math.}, 216(2):421--443, 2019.

\bibitem{MNS}
D.~Maximo, I.~Nunes, and G.~Smith.
\newblock Free boundary minimal annuli in convex three-manifolds.
\newblock {\em J. Differential Geom.}, 106(1):139--186, 2017.

\bibitem{MSY}
W.~Meeks, L.~Simon, and S.~Yau.
\newblock Embedded minimal surfaces, exotic spheres, and manifolds with
  positive {R}icci curvature.
\newblock {\em Ann. of Math.}, 116(3):621--659, 1982.

\bibitem{MillotSire}
V.~Millot and Y.~Sire.
\newblock On a fractional {G}inzburg-{L}andau equation and 1/2-harmonic maps
  into spheres.
\newblock {\em Arch. Ration. Mech. Anal.}, 215(1):125--210, 2015.

\bibitem{Nitsche}
J.~Nitsche.
\newblock Stationary partitioning of convex bodies.
\newblock {\em Arch. Rational Mech. Anal.}, 89(1):1--19, 1985.

\bibitem{Petrides_disks}
R.~Petrides.
\newblock Non planar free boundary minimal disks into ellipsoids.
\newblock {\em arXiv:2304.12111}, 2023.

\bibitem{Poincare_geod_conj}
H.~Poincar\'{e}.
\newblock Sur les lignes g\'{e}od\'{e}siques des surfaces convexes.
\newblock {\em Trans. Amer. Math. Soc.}, 6(3):237--274, 1905.

\bibitem{SS}
F.~Smith.
\newblock On the existence of embedded minimal 2-spheres in the 3-sphere,
  endowed with an arbitrary {R}iemannian metric.
\newblock {\em Phd thesis, Supervisor: Leon Simon, University of Melbourne},
  1982.

\bibitem{Song_inf}
A.~Song.
\newblock Existence of infinitely many minimal hypersurfaces in closed
  manifolds.
\newblock {\em Ann. of Math. (2)}, 197(3):859--895, 2023.

\bibitem{Struwe_fb}
M.~Struwe.
\newblock On a free boundary problem for minimal surfaces.
\newblock {\em Invent. Math.}, 75(3):547--560, 1984.

\bibitem{Struwe_half}
M.~Struwe.
\newblock Plateau flow or the heat flow for half-harmonic maps.
\newblock {\em arXiv:2202.02083}, 2022.

\bibitem{SWZ_mult1}
A.~Sun, Z.~Wang, and X.~Zhou.
\newblock Multiplicity one for min-max theory in compact manifolds with
  boundary and its applications.
\newblock {\em arXiv:2011.04136}, 2020.

\bibitem{Wang_math_ann}
Z.~Wang.
\newblock Min-max minimal hypersurface in manifolds with convex boundary and
  {${\rm Ric}\ge 0$}.
\newblock {\em Math. Ann.}, 371(3-4):1545--1574, 2018.

\bibitem{Wang_inf}
Z.~Wang.
\newblock Existence of infinitely many free boundary minimal hypersurfaces.
\newblock {\em arXiv:2001.04674}, 2020.

\bibitem{WangZhou}
Z.~Wang and X.~Zhou.
\newblock Existence of four minimal spheres in {$S^3$} with a bumpy metric.
\newblock {\em arXiv:2305.08755}, 2023.

\bibitem{Wett1}
J.~Wettstein.
\newblock Uniqueness and regularity of the fractional harmonic gradient flow in
  {$S^{n-1}$}.
\newblock {\em Nonlinear Anal.}, 214:Paper No. 112592, 48, 2022.

\bibitem{Wett2}
J.~Wettstein.
\newblock Half-harmonic gradient flow: aspects of a non-local geometric {PDE}.
\newblock {\em Math. Eng.}, 5(3):Paper No. 058, 38, 2023.

\bibitem{White_deg2}
B.~White.
\newblock The space of minimal submanifolds for varying {R}iemannian metrics.
\newblock {\em Indiana Univ. Math. J.}, 40(1):161--200, 1991.

\bibitem{Yau_problems}
S.~Yau.
\newblock Problem section.
\newblock In {\em Seminar on {D}ifferential {G}eometry}, volume 102 of {\em
  Ann. of Math. Stud.}, pages 669--706. Princeton Univ. Press, Princeton, N.J.,
  1982.

\bibitem{Zhou_mult1}
X.~Zhou.
\newblock On the multiplicity one conjecture in min-max theory.
\newblock {\em Ann. of Math. (2)}, 192(3):767--820, 2020.

\bibitem{Zhou_ICM}
X.~Zhou.
\newblock Mean curvature and variational theory.
\newblock {\em Proceedings of the ICM}, 2022.

\bibitem{Zoll}
O.~Zoll.
\newblock Ueber {F}l\"{a}chen mit {S}charen geschlossener geod\"{a}tischer
  {L}inien.
\newblock {\em Math. Ann.}, 57(1):108--133, 1903.

\end{thebibliography}

\bibliographystyle{abbrv}

\bigskip

\end{document}